\documentclass[nofonts]{dgleich-article}
\usepackage{dgleich-math}
\usepackage{txfonts}
\usepackage{subcaption}
\usepackage{adjustbox}
\newcommand{\mSigma}{\mat{\Sigma}}
\usepackage{url}
\usenatbib
\author{David F.~Gleich}
\title{Better than best low-rank approximation with the singular value decomposition}
\begin{document}
\maketitle
	
	\begin{abstract}
	The Eckhart-Young theorem states that the best low-rank approximation of a matrix can be constructed from the leading singular values and vectors of the matrix.  Here, we illustrate that the practical implications of this result crucially depend on the organization of the matrix data. In particular, we will show examples where a rank $2$ approximation of the matrix data in a different representation more accurately represents the entire matrix than a rank $5$ approximation of the original matrix data -- even though both approximations have the same number of underlying parameters. Beyond images, we show examples of how flexible orientation enables better approximation of time series data, which suggests additional applicability of the findings. Finally, we conclude with a theoretical result that the effect of data organization can result in an unbounded improvement to the matrix approximation factor as the matrix dimension grows. 
	\smallskip 
	
		\noindent {\hfill \emph{In memory of Gene H. Golub for what would have been his 23rd birthday.} \hfill}
	\end{abstract}

\section{Introduction}

In data analysis, we often approximate a complex data set with something easier to understand or manipulate. When the data can be organized into a matrix, then a common paradigm is to compute a low-rank -- and hence low-parameter -- approximation of the data. If $\mX$ is the $m \times n$ matrix of data  %that might represent pixels of an image, or measurements of genes in various scenarios, or time series, 
then we seek to find a matrix $\mY$ where $\mY$ is \emph{low-rank}, as in the following problem 
\[ \text{ find } \mY \text{ with rank} \le k \text{ to minimize } \sum\nolimits_{ij} |X_{ij} - Y_{ij}|^2. \]
%Here, $\normof{\mX-\mY}^2$ is simply the sum of squared differences $\sum_{ij} |X_{ij} - Y_{ij}|^2$. 
The idea is that $\mY$ should highlight the most dominant effects in $\mX$ and this is closely related to principal components analysis (PCA).\footnote{For additional examples, and a critique of findings in PCA, see \citet{Shinn-2023-pca}.}
%Well-known examples of low-rank matrix approximation in practice examples include: gene analysis~\cite{Orly-Alter}, motion analysis~\cite{XXX},
	%latent semantic indexing~\cite{XXX}, the netflix problem~\cite{XXX}, hubs-and-authorities analysis~\cite{XXX}.
	
The solution to this problem is given by the leading singular vectors of a matrix~\cite{Eckart-1936-rank-r}. If $\vu_1, \vu_2, \ldots$ are the leading left singular vectors of $\mX$, $\vv_1, \vv_2, \ldots$ are the leading right singular vectors, and $\sigma_1 \ge \sigma_2 \ldots \ge 0$ are the leading singular values in descending order then the solution $\mY = \sum_{i=1}^r \sigma_i \vu_i \vv_i^T $.  Algorithms to compute the $\vu$s, $\vv$s, and $\sigma$s are available in all technical computing languages (Matlab, Julia, Python, Fortran, R, etc.) for both dense and sparse matrices $\mX$ helping to promote and sustain the ubiquity of such analysis as well as its use in upstream statistical scenarios such as PCA.\footnote{This problem is stated in the so-called Frobenius norm for simplicity. Those familiar with the area should recognize that the results will translate to any unitary invariant norm as in the \citet{Mirsky-1960-svd} generalization.} 

An important and under-appreciated aspect of this result is that it depends critically on how the data in $\mX$ are organized into a matrix. Consider the following two tables of numbers represented as matrices
\[ \mX_1 = \left[ \begin{array}{ccc|ccc} -3 & -6 & -9 & -12 & -15 & -18 \\ 
-2 & -4 & -6 & -8 & -10 & -12 \\ 
-1 & -2 & -3 & -4 & -5 & -6 \\ 
\midrule
1 & 2 & 3 & 4 & 5 & 6 \\ 
2 & 4 & 6 & 8 & 10 & 12 \\ 
3 & 6 & 9 & 12 & 15 & 18 
\end{array} \right] 
\qquad \mX_2 = \left[ \begin{array}{c|c|c|c}
 -3 & 1 & -12 & 4 \\ 
-2 & 2 & -8 & 8 \\ 
-1 & 3 & -4 & 12 \\ 
-6 & 2 & -15 & 5 \\ 
-4 & 4 & -10 & 10 \\ 
-2 & 6 & -5 & 15 \\ 
-9 & 3 & -18 & 6 \\ 
-6 & 6 & -12 & 12 \\ 
-3 & 9 & -6 & 18
\end{array} \right]. \] 
	\marginnote[-5ex]{David F.~Gleich, Purdue University, Computer Science.
	\texttt{dgleich@purdue.edu}. Supported in part by NSF IIS-2007481, DOE DE-SC0023162 and IARPA AGILE.}
They contain exactly the same set of data, but are organized differently. In $\mX_2$, each square block of $\mX_1$ is organized into a long vector. The best rank-1 approximations of the two matrices are 
\[ \mY_1 = \left[ \begin{array}{ccc|ccc} -3 & -6 & -9 & -12 & -15 & -18 \\ 
-2 & -4 & -6 & -8 & -10 & -12 \\ 
-1 & -2 & -3 & -4 & -5 & -6 \\ 
\midrule
1 & 2 & 3 & 4 & 5 & 6 \\ 
2 & 4 & 6 & 8 & 10 & 12 \\ 
3 & 6 & 9 & 12 & 15 & 18 
\end{array} \right] 
\qquad \mY_2 = \left[ \begin{array}{c|c|c|c}
-3.1 & 3.1 & -7.5 & 7.5 \\ 
-3.1 & 3.1 & -7.5 & 7.5 \\ 
-3.1 & 3.1 & -7.5 & 7.5 \\ 
-4.2 & 4.2 & -9.9 & 9.9 \\ 
-4.2 & 4.2 & -9.9 & 9.9 \\ 
-4.2 & 4.2 & -9.9 & 9.9 \\ 
-5.2 & 5.2 & -12.3 & 12.3 \\ 
-5.2 & 5.2 & -12.3 & 12.3 \\ 
-5.2 & 5.2 & -12.3 & 12.3
\end{array} \right] 
\]

The matrix $\mX_1$ is rank-1, so it can be represented without any error as a rank-1 matrix. Despite representing exactly the same numbers, the matrix $\mX_2$ is rank-4, consequently it must be approximated with some error as a rank-1 matrix. Note that $\normof[F]{\mX_2 - \mY_2}^2 = 378$, which is not close to zero.  Note further that the \emph{inexact} rank-1 representation of $\mY_2$ has 13 parameters, compared with 12 parameters for the exact rank-1 representation of $\mX_1$. 

Consequently, changing the representation of a matrix impacts the results of low-rank approximation with the singular value decomposition. While this statement is known~\cite{VanLoan-1993-kronecker}, it is frequently overlooked when comparing to low-rank approximation via the SVD. We show that the impact can be dramatic. % and reveal new connections between the SVD in images and the discrete cosine transformation underlying the JPEG compression scheme. 

\section{A case study in images}
A famous example demonstrating low-rank approximation via the SVD is to apply it to a matrix representation of an image (e.g.~\citet{moler2004-matlab}). We illustrate this with a 240,000 pixel image in Figure~\ref{fig:svd-tiles}(A) and (C,D,E) for rank $5$, $15$, $25$ approximations. The rank-5 approximation (5000 parameters) is not useful.  The rank-15 approximation (15000 parameters) smears details substantially. The rank-25 approximation (25000 parameters) is passable.  

If, instead, the image is represented as a matrix-of-tiles (illustrated in Figure~\ref{fig:svd-tiles}(B)), then we have a scenario with two organizations of the same numbers again. We find that using the matrix-of-tiles approach gives better approximations -- both subjectively and quantitatively as illustrated in Figure~\ref{fig:svd-tiles}(F,G,H). The rank-2 approximation (5000 parameters) is almost useful whereas both the rank-3 (7500 parameters) and rank-4 approximations (10000 parameters) are as good or better than the rank-25 approximation of the original matrix -- despite using only a small fraction of the parameters.  

\paragraph{A ubiquitous finding} This finding is ubiquitous across a pool of over 48,000 images from the validation set in the ImageNet database~\cite{Deng2009-imagenet} where there was some tiled approximation that was better at achieving a 5 or 10\% relative error in 98.6\% of images. 

To determine this result, for each image in the validation set of ImageNet, we computed a tiled approximation at tile sizes 7, 11, 15, 19, 23, 27, 31, 35, 39, 43. Then we checked if any of these tile sizes gave a lower parameter 5\% or 10\% relative error approximation (measured in the Frobenius norm). This means that for 1.4\% of images, the traditional SVD approximation was superior in all of these tests, and for the overwhelming majority, the tiled approximation was superior. The choice of tile sizes was designed to avoid any potentially confounding issues with JPEG compression and its 8-by-8 tile based compression. Additional ad hoc tests on uncompressed source data showed that this result is unlikely to have been produced as an artifact of JPEG compression, although that possibility has not been completely excluded. 

%\footnote{Full details of the comparison are in Section~\ref{sec:details-imagenet}.} 

\begin{fullwidthfigure}[t]
%\begin{subfigure}
\includegraphics[width=0.5\linewidth]{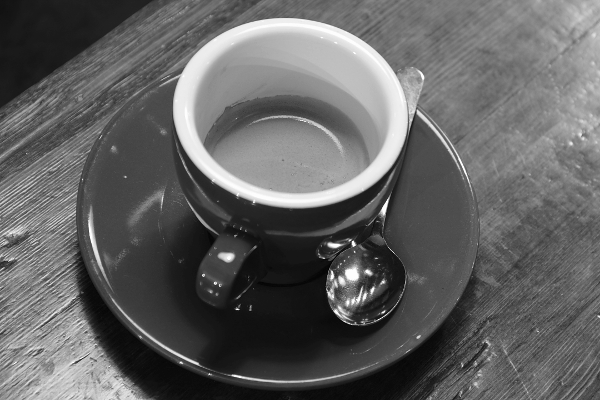}%
\adjustbox{valign=b}{
\begin{tabular}{@{}c@{}}
\includegraphics[width=0.5\linewidth]{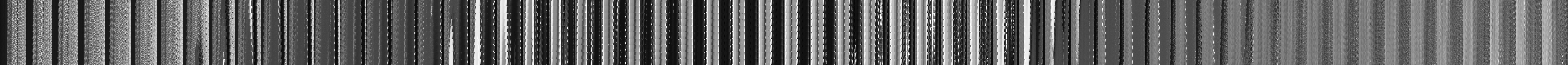} \\
$\uparrow$ \\ 
\includegraphics[width=0.4\linewidth]{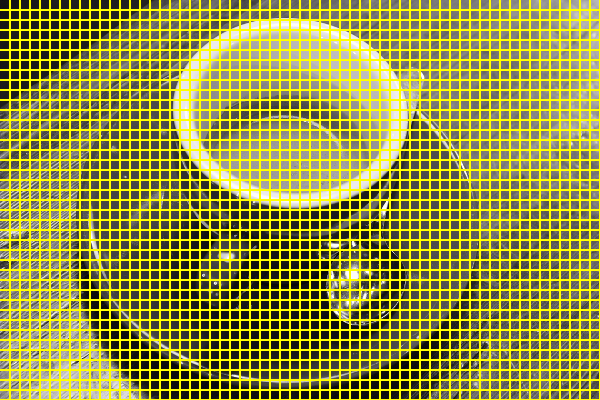} \\ 
\end{tabular}}
\begin{minipage}{0.5\linewidth} \raggedright \footnotesize 
(A) A 400 $\times$ 	600 matrix that represents an image of a coffee cup with 240,000 parameters.
\end{minipage}%
\begin{minipage}{0.5\linewidth} \raggedright \footnotesize 
(B) A 100 $\times$ 	2400 matrix of tiles that represents each $10 \times 10$ tile of the image in a column. The matrix still has 240,000 parameters.
\end{minipage}%
\bigskip

\includegraphics[width=0.33\linewidth]{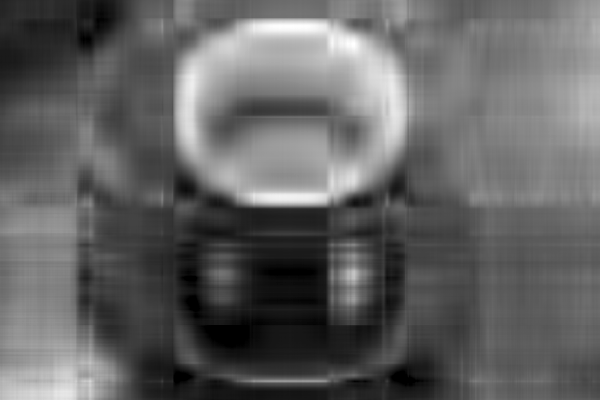}%
\includegraphics[width=0.33\linewidth]{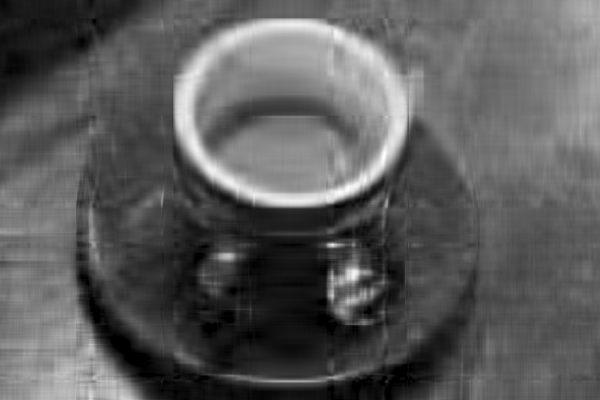}%
\includegraphics[width=0.33\linewidth]{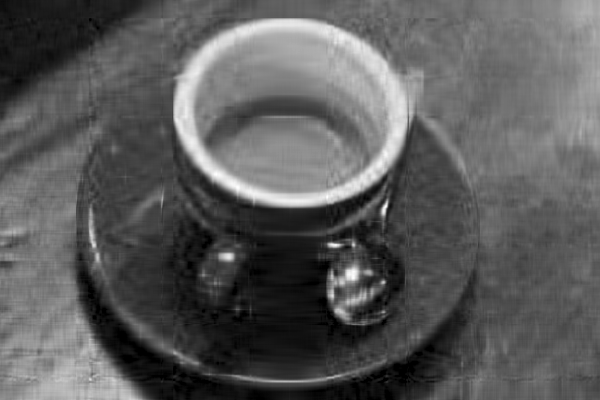}
\begin{minipage}{0.33\linewidth} \raggedright \footnotesize 
(C) A rank-5 approximation of the image matrix \newline with 22\% error and 5000 parameters
\end{minipage}%
\begin{minipage}{0.33\linewidth} \raggedright \footnotesize 
(D) A rank-15 approximation of the image matrix\newline with 15\% error and 15000 parameters
\end{minipage}%
\begin{minipage}{0.33\linewidth} \raggedright \footnotesize 
(E) A rank-25 approximation of the image matrix\newline with 12\% error and 25000 parameters
\end{minipage}%
\bigskip

\includegraphics[width=0.33\linewidth]{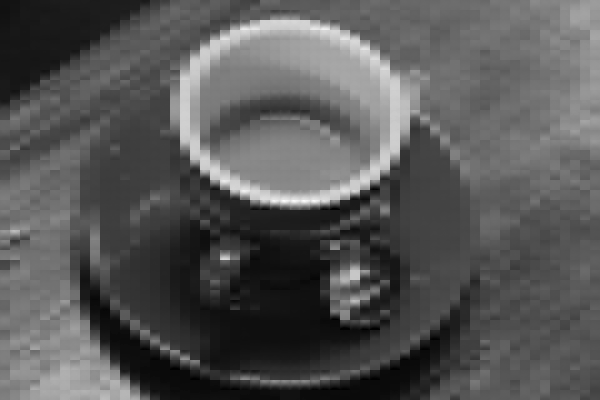}%
\includegraphics[width=0.33\linewidth]{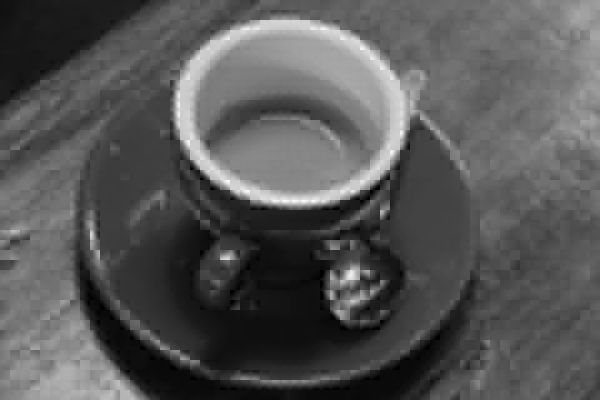}%
\includegraphics[width=0.33\linewidth]{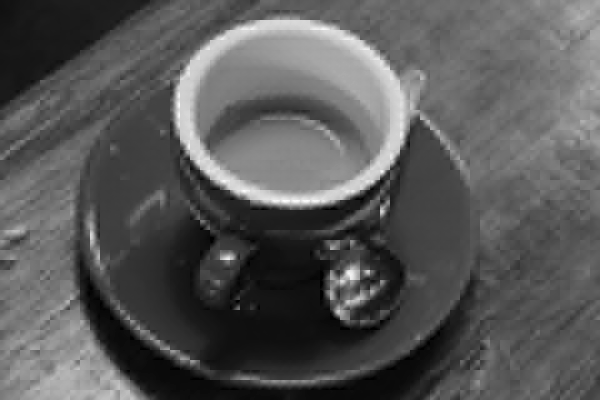}
\begin{minipage}{0.33\linewidth} \raggedright \footnotesize 
(F) A rank-2 approximation of the matrix of tiles\newline with 15\% error and 5000 parameters
\end{minipage}%
\begin{minipage}{0.33\linewidth} \raggedright \footnotesize 
(G) A rank-3 approximation of the matrix of tiles\newline with 13\% error and 7500 parameters
\end{minipage}%
\begin{minipage}{0.33\linewidth} \raggedright \footnotesize 
(H) A rank-4 approximation of the matrix of tiles\newline with 12\% error and 10000 parameters
\end{minipage}%

\caption{An example of how re-arranging the image data from a matrix (A) into a matrix of tiles (B) produces better approximations with fewer parameters. For instance (C) vs (F) shows a large reduction in error with the same parameters whereas (E) vs (H) shows a large reduction in parameters at the same error. Comparing (D) and (G) shows reduction in both error and parameters result from approximating the matrix in (B) compared with (A). } 
%\end{subfigure}
%\caption{Test}
\label{fig:svd-tiles}	
\end{fullwidthfigure}

\clearpage

\section{A case study in temporal data}
The result that \emph{better} approximations of data are possible with reshaping the matrix also exists for multiple samples of temporal data. In this case, we consider a matrix of daily COVID-19 positivity rates from each of the 50 US states for 150 days. This can be simply represented as a 50-by-150 matrix $\mX_1$, where each state is a row. The values within the row are consecutively ordered by the date of their measurement. Consider now, if we divide the days up into three groups of 50 consecutive days, then we have $\mX_1 = \bmat{ \mG_1 & \mG_2 & \mG_3 }$ where each matrix $\mG_i$ is a group of 50 consecutive days.  Then we can reorganize into \[ \mX_2 = \bmat{ \mG_1 \\ \mG_2 \\ \mG_3 }. \] This gives a 150-by-50 matrix. If we compute the rank-2 SVD approximation of $\mX_2$, it is quantatively and qualitatively better than the rank-2 SVD approximation of $\mX_1$ as shown in Figure~\ref{fig:covid-cases}.  In this case, doing a rank-2 approximation of the reorganized data corresponds to a piecewise-linear approximation of the data in three distinct time regimes. This allows the method to better approximate changes in the trajectories. 

\begin{fullwidthfigure}
	\includegraphics[width=0.33\linewidth]{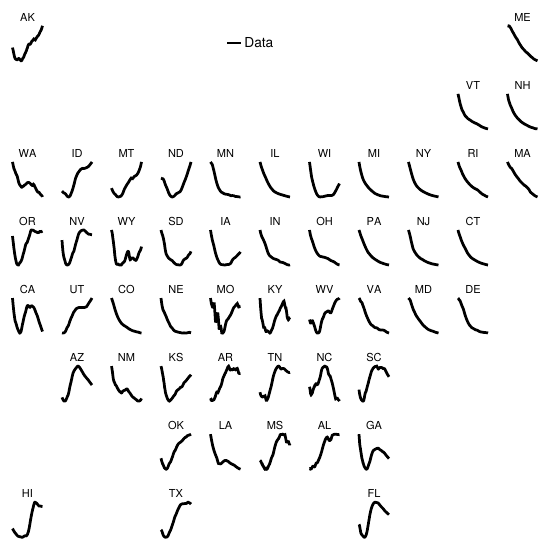}%
	\includegraphics[width=0.66\linewidth]{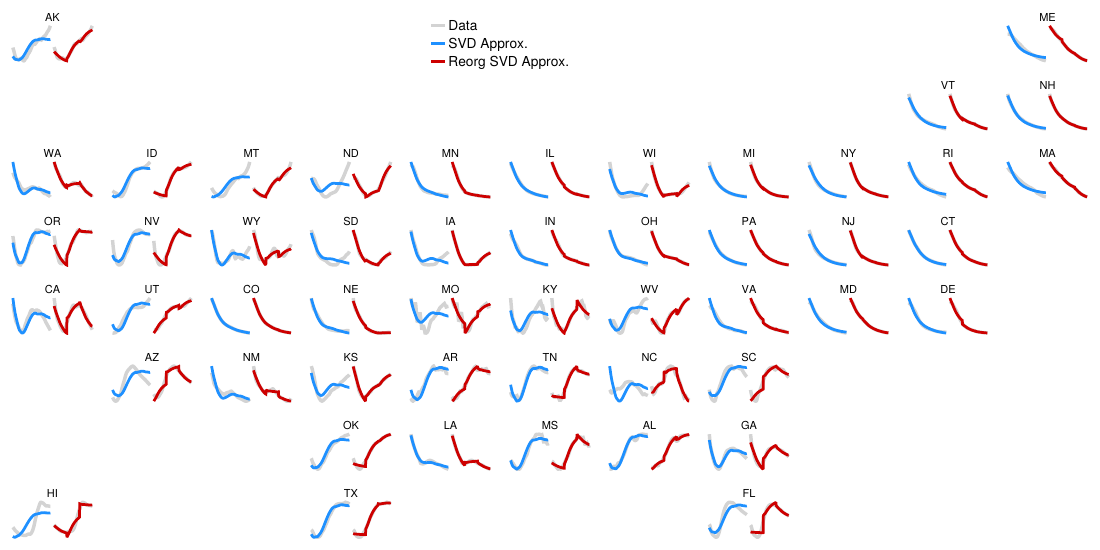}
\begin{minipage}{0.33\linewidth}\raggedright\footnotesize
(A) 150 days of data for 50 states (7500 parameters)
\end{minipage}%	
\begin{minipage}{0.66\linewidth}\raggedright\footnotesize 
(B) In blue, a rank 2, 300-parameter SVD approximation with 6.7\% error. In red, a rank 2, 300-parameter reorganized SVD approximation with 2.7\% error. Ground truth is in grey behind each. 
\end{minipage}

\caption{ (A) Moving average COVID-19 positivity rates for the 50 US states for 150 days starting from May 17, 2020. (B) If we view this as a 50 $\times$ 150 matrix in terms of states-by-days and use the optimal rank-2 approximation, we get 6.7\% error with 300 parameters and the approximation shown in blue. The original data is shown as shaded. If we view this as a 150 $\times$ 50 matrix by the reorganization described in the text and use the optimal rank-2 approximation, we get 2.7\% error with the same number of parameters (300). Moreover, the approximation is qualitatively better -- consider Arizona (AZ), Hawaii (HI), Kentucky (KY), Montana (MT), North Dakota (ND), North Carolina (NC) }		
\label{fig:covid-cases}	
\end{fullwidthfigure}

\paragraph{Details}
The data on COVID-19 positivity rates is from \url{https://api.covidtracking.com/v1/states/daily.csv}. Then we divided the total number of positive results by the total number of tests to get a positivity rate. We then formed a 7 day moving average of measurements starting from May 17, 2020 until October 13, 2022 (150 days) for each of the 50 US states for 150 days. (The 7-day moving averages on May 17 reflects May 11-May 17.) 
The plot is normalized to show the maximum dynamic range for each state. 

%\begin{figure}
%\begin{subfigure}
%
%	\caption{test}
%\end{subfigure}
%\caption{(A) shows a 50 $\times$ 150 matrix of state-by-day COVID positivity rate measurements (B) shows the data re-organized into a 150 $\times$ 50 matrix following the description in the text. The resulting .}

%\end{figure}

%
%\section{A case study in gene data}
%Another prominent class of matrix data are gene expression records. Here, we study a set of example bulk scRNA-seq data regarding lung cancer from the Ecotyper package~\cite{Steen-2021-ecotyper-example,Luca-2021-ecotyper}. This yields an original matrix $\mX_1$ with 19776 rows and 250 columns of raw expression information. If we organize this matrix into 
%

\section{A theoretical bound on inverses of some tridiagonal matrices}
The results thus far have been empirical. We show a scenario where the best reorganized approximation \emph{gets better} as the matrix grows in size. 

\paragraph{A simple case: the identity matrix}
Consider the identity matrix. Let $\mI^{(1)}$ be any of the best rank 1 approximations of the identity matrix results. Then we have a Frobenius norm difference of $\sum_{ij} |I_{ij} - I_{ij}^{(1)}|^2 = n-1$. In contrast, if we reorganize $\mI$ into a matrix $\mB$ such that the diagonal of $\mI$ becomes the first column of $\mB$, then $\mB$ is exactly a rank 1 matrix. Consequently, we can find a difference as large as we like between a reorganized approximation and the original approximation. Admittedly, this case is probably too simple to be useful. 

\paragraph{A more involved case: strong diagonals} The idea for a more involved case is to show that the essence of the result for the identity matrix would hold for the inverse of a class of tridiagonal matrices. These inverses of diagonally dominant tridiagonal matrices will have a strong diagonal term and decaying off diagonals, a common pattern in many more complicated matrices from applications as well~\cite{Benzi-2013-decay}. This scenario was constructed such that the values are straightforward to compute and bound. We did not attempt to produce the tightest bounds.  

In this proof, we will use both the squared Frobenius norm: $\normof[F]{\mA} = \sqrt{\sum_{ij} A_{ij}^2}$ as well as the spectral norm $\normof{\mA}$, which is the largest singular value of $\mA$.  We also use $\rho(\mA)$ to denote the spectral radius of $\mA$, which is the magnitude of the largest eigenvalue. 
%is that start with a nonsingular matrix with a strong diagonal, i.e. something tha Using a rank 1 approximation of this matrix is going to be bad as This is illustrated in Figure~\ref{fig:tridiagonal approximation}. We can formalize the result in the following theorem.

\begin{theorem} \label{thm:main} 
Let \[ \mX^{-1} = \bmat{ \gamma & \beta\gamma \\ \alpha\gamma & \gamma(1+\alpha \beta) & \beta\gamma 
\\ & \alpha\gamma & \gamma(1+\alpha \beta) & \beta\gamma \\ 
& & \ddots & \ddots & \ddots \\
& & & \alpha\gamma & \gamma(1+\alpha \beta) } \]  	
where $|\alpha| < 1$, $|\beta|< 1$, and $\gamma > 0$. Let the matrix size be $n \times n$. 
Then let $\mX^{(1)}$ be the best rank-1 approximation of $\mX$, we have $\sum_{ij} (X_{ij} - X^{(1)}_{ij})^2 \ge (\sum_{ij} X_{ij}^2) - \omega_X^2$ for a constant $\omega_X$ independent of $n$. If we organize $\mX$ into a matrix $\mB$ where the first column of $\mB$ corresponds to the diagonal of $\mX$,  then if $\mB^{(1)}$ is the best rank-1 approximation of $\mB$, we have $\sum_{ij} (B_{ij} - B^{(1)}_{ij})^2 \le (\sum_{ij} X_{ij}^2) - \omega_1 n - \omega_2(n)$, where $\omega_1$ is a constant and $\omega_2(n)$ converges geometrically fast to a constant. This improves with $n$ and dominates the other approximation for sufficiently large $n$.
\end{theorem}

The proof of this result is fairly straightforward. 
\begin{compactitem}
\item Lemma~\ref{lem:lu} shows a detail on how we constructed the matrix $\mX^{-1}$ as the product of two bidiagonal matrices to create an easy-to-analyze problem.
\item Lemma~\ref{lem:spectral} gives a characterization of the spectral norm of $\mX$, which is a bound on the largest singular value. This shows that the best rank-1 approximation cannot give more than a constant change to the overall sum of squared elements. This uses a Gershgorin disk bound. 
\item Lemma~\ref{lem:inverse} works through the exact form of all the entries of $\mX$. 
\item Lemma~\ref{lem:Tn} works through the approximation implied by just approximating the main diagonal.  	
\end{compactitem}

\begin{lemma} \label{lem:lu} 
Let $\mL$ be the lower bidiagonal matrix with $1$ on the diagonal and $\alpha$ on the sub-diagonal. Let $\mU$ be the upper bidiagonal matrix with $\gamma > 0$ on the diagonal and $\beta \gamma$ on the super-diagonal. Let $|\alpha| < 1$ and $|\beta|< 1$. Let $n$ be the dimension of the matrix. Then $\mA = \mX^{-1} = \mL \mU$.	
\end{lemma}
\begin{proof}
This is direct computation. We have 
\[ \mL = \bmat{1 \\ \alpha & 1 \\ & \alpha & 1 \\ & & \ddots & \ddots \\ & & & \alpha & 1 } 
\qquad	\text{ and } \qquad 
\mU = \gamma \bmat{ 1 & \beta \\ & 1 & \beta \\ & & \ddots & \ddots \\ & & & 1 & \beta \\ & & & & 1 }. 
\]
\end{proof}

\begin{lemma} \label{lem:spectral}
In the same scenario as Theorem~\ref{thm:main} and Lemma~\ref{lem:lu}, consider the matrix $\mX = \mA^{-1} = \mU^{-1} \mL^{-1}$. Then the spectral operator norm, i.e. the largest singular value, satisfies the bound  
\[ \normof{\mX} = \normof{\mA^{-1}} \le \sqrt{\frac{1}{1-\alpha^2}(2/(1-|\alpha |) - 1)} \sqrt{\frac{(1/\gamma)^2}{1-\beta^2}(2/(1-| \beta |)-1)}. \] 
%In addition, we have 
%\[ \mA^{-1}_{i,i} \ge \max\left(\frac{1}{\gamma}(1 + \alpha\beta), \frac{1}{\gamma}\right) \] 	
\end{lemma}
\begin{proof}
	The proof follows by repeated application of the geometric series or Neumann series, combined with a Gershgorin disk result on the extremal eigenvalue. Recall that 
\[ 1+\alpha+\alpha^2+\ldots = 1/(1-\alpha) \text{ when } |\alpha| < 1\text{ and } \mI + \mA + \mA^2 + \ldots = (\mI - \mA)^{-1} \text{ when } \rho(\mA) < 1. \] 
First note that
\[ \mL^{-1} = \bmat{1 \\ -\alpha & 1 \\ (-\alpha)^2 & -\alpha & 1 \\  \vdots & \ddots & \ddots & \ddots \\
(-\alpha)^{n-1} & \cdots & (-\alpha)^2 & -\alpha & 1 }, \] 
which follows by using the matrix geometric. 
Likewise 
\[ \mU^{-1} = (1/\gamma) 
\bmat{
1 & -\beta & (-\beta)^2 & \cdots & (-\beta)^{n-1} \\
& \ddots & \ddots & \ddots & \vdots \\
& & 1 & -\beta & (-\beta)^2 \\
& & & 1 & -\beta \\
& & & & 1 } \]

Now, we use the simple upper-bound 
$\normof{\mA^{-1}} \le \normof{\mU^{-1}} \normof{\mL^{-1}}$. 
In order to estimate both $\normof{\mL^{-1}}$ and $\normof{\mU^{-1}}$, we use a Gersgorin disk bound on the eigenvalues of ${\mL^{-1}}^T {\mL^{-1}}$  and ${\mU^{-1}}^T  {\mU^{-1}}$.

Let $C_k = \frac{1-\alpha^{2k}}{1-\alpha^2}$, then 
\[ {\mL^{-T}} {\mL^{-1}} = \bmat{
C_n & -\alpha C_{n-1} & (-\alpha)^2 C_{n-2} & \cdots & (-\alpha)^{n-1} C_{1} \\
-\alpha C_{n-1} & C_{n-1} & -\alpha C_{n-2} & \cdots & (-\alpha)^{n-2}C_1\\
\vdots & \ddots & \ddots & \ddots & \vdots \\
(-\alpha)^{n-1} C_1 & \cdots & \cdots & \alpha C_{1} & C_1 \\ 
}. \] 

Likewise let $D_k = \frac{1-\beta^{2k}}{1-\beta^2}$, then 
\[ {\mU^{-T}} \mU^{-1} = (1/\gamma)^{2} \bmat{
D_1 & (-\beta) D_1 & (-\beta)^2 D_1 & \cdots & (-\beta)^{n-1} D_1 \\
(-\beta) D_1 & D_2 & (-\beta) D_2 & \cdots & (-\beta)^{n-2} D_2 \\
(-\beta)^2 D_1 & \vdots & \ddots & \ddots & \vdots \\
\vdots & \vdots & \ddots & D_{n-1}  & (-\beta) D_{n-1}\\ 
(-\beta)^{n-1} D_1 & \cdots & \cdots & (-\beta) D_{n-1} & D_n
}. \]

\[ [\mL^{-T} \mL^{-1}]_{ij} = (-\alpha)^{|i-j|}C_{n-\max(i,j)+1}  \] 
\[ [\mU^{-T} \mU^{-1}]_{ij} = (1/\gamma)^2 (-\beta)^{|i-j|}D_{\min(i,j)}  \] 

To apply the Gersgorin disk result on these matrices to get an upper bound on the spectral operator norms of $\mL^{-1}$ and $\mU^{-1}$, we simply note that we can replace $C_k$ and $D_k$ by $C_{\infty}$ and $D_{\infty}$. (Note that both terms are positive.) Then in a row in the middle, we have (after applying absolute value signs) the upper bound: 
\[ \bmat{ 
& & & & \text{diagonal} \\
& & & & \downarrow  \\
\cdots &  |\beta|^{3} D_{\infty} & |\beta|^{2} D_{\infty} & |\beta|^{1} D_{\infty} &  D_{\infty} & |\beta|^{1} D_{\infty} &  |\beta|^{2} D_{\infty} &  |\beta|^{3} D_{\infty} & \cdots \\
} \] 
If we simply extended this to an infinite size matrix, we have the upper bound from Gershgorin disks as 
\[ \rho(\mL^{-T} \mL^{-1}) \le C_{\infty} (2/(1-|\alpha|) - 1). \] 
\[ \rho(\mU^{-T} \mU^{-1}) \le (1/\gamma)^2 D_{\infty} (2/(1-|\beta|) - 1). \] 
The lemma follows because $C_{\infty} = 1/(1-\alpha^2)$ and $D_{\infty} = 1/(1-\beta^2)$ and both $\normof{\mL^{-1}} \le \sqrt{\rho(\mL^{-T} \mL^{-1})}$, $\normof{\mU^{-1}} \le \sqrt{\rho(\mU^{-T} \mU^{-1})}$.
\end{proof}

This lemma gets us half the way to the result. We know that the best rank 1 approximation of $\mX$ is a constant that is independent of $n$. Hence, $\normof[F]{\mX - \mX^{(1)}}^2 \ge \normof[F]{\mX}^2 - \omega_X^2$ for some constant $\omega_X$.\footnote{The constant $\omega_X$ is equal to the value of $\normof{\mX}$, which, as shown in Lemma~\ref{lem:spectral}, does not scale with $n$. This is also the largest singular value of $\mX$. This result follows because the Frobenius norm is unitarily invariant, so $\normof[F]{\mX - \mX^{(1)}}^2 = \normof[F]{\mSigma - \mSigma^{(1)}}^2$, where $\mSigma$ is the diagonal matrix of singular values and $\mSigma^{(1)}$ is the first rank-1 factor in the SVD. Thus, we have the bound $\normof[F]{\mX}^2 - \omega_X^2$.} 

To determine that the diagonal alone is a better approximation, we need to make sure the elements of the diagonal do not become too small. The next lemma gives that information. 

\begin{lemma} \label{lem:inverse}
Let $\mA = \mL \mU$ with the same conditions as Theorem~\ref{thm:main}, Lemma~\ref{lem:lu}. Then 
\[ \mX = \mA^{-1} = \frac{1}{\gamma} \bmat{
S_{n} & (-\beta) S_{n-1} & (-\beta)^2 S_{n-2} & \cdots & (-\beta)^{n-1} S_{1} \\
(-\alpha) S_{n-1} & S_{n-1} & (-\beta) S_{n-2} & \cdots & (-\beta)^{n-2} S_{1} \\
(-\alpha)^2 S_{n-2} & (-\alpha) S_{n-2} & S_{n-2} & \cdots & (-\beta)^{n-3} S_{1} \\
\vdots & \vdots & \ddots & \ddots  & \vdots \\ 
(-\alpha)^{n-2} S_2 & \cdots & \cdots &  S_2  & (-\beta) S_1 \\
(-\alpha)^{n-1} S_1 & \cdots & \cdots & (-\alpha) S_1  & S_1 \\
}
\] 
where $S_i = \sum_{j=0}^{i-1} (\alpha \beta)^i$. 
\end{lemma}
\begin{proof}
This follows from direct computation of elements of $(1/\gamma) \mU^{-1} \mL^{-1}	$. 
\end{proof}

\begin{lemma} \label{lem:Tn}
Let $\alpha, \beta, \gamma, \mX$, and $S_i$ be as in Theorem~\ref{thm:main}, Lemma~\ref{lem:inverse}.
Let $T_n = \sum_{i=1}^n 	S_i^2$ and let $\delta = \alpha \beta$. Then 
\[ T_n = \frac{1}{(1-\delta)^2}\left( n - \frac{2\delta(1-\delta^n)}{1-\delta} + \frac{\delta^2(1-\delta^{2n})}{1-\delta^2} \right) \] %\text{ and } \]
% \[ \sum_{ij} X_{ij}^2 = \frac{1}{\gamma^2} \left(T_n + \sum_{i=1}^{n-1} (\alpha^{2i} + \beta^{2i}) T_{n-i}\right). \] 
\end{lemma}
\begin{proof}
Note that relationship between the diagonals of $\mX$ and scaled versions of $T_n$. 
The structure of $T_n$ follows from using a truncated geometric series to compute $S_i = (1-\delta^i)/(1-\delta)$, then two truncated geometric series to handle $T_n = \sum_{i=1}^n S_i^2 = \sum_{i=1}^n (1 - 2 \delta^i + \delta^{2i})/(1-\delta)^2$. 
This rest follows from organizing the matrix $\mX$ into diagonals and using the scaled sum. 
\end{proof}

Consider the implication of this result for Theorem~\ref{thm:main}. If we organize $\mX$ so that the diagonal becomes the first column, then the best rank-1 approximation has to be \emph{at least as good} as that single column. Hence,%\footnote{This follows because we can choose the particular rank 1 matrix $B_{ij}^{(1)}$ whose first column is $S_n, S_{n-1}, \ldots$ to remove the impact of those entries. Thus $\sum_{ij} (B_{ij} - B_{ij}^{(1)})^2 \le \sum_{ij} B_{ij}^2 - 1/\gamma^2 T_n$. }
\[ \normof[F]{\mB - \mB^{(1)}}^2 \le \normof[F]{\mB}^2 - \frac{1}{\gamma^2} T_n. \] 
But we have that $T_n$ is almost linear function in $n$, that is, $\frac{1}{\gamma^2} T_n = \omega_1 n + \omega_2(n)$. In this case, $\omega_2(n)$ converges to a constant geometrically fast.  This completes the final piece of Theorem~\ref{thm:main}.

% that the difference between the approximation 

%The matrix $\mX_1$ is rank-1 with $\vu_1 = \bmat{ -3 & -2 & -1 & 1 & 2 & 3 }, \vv_1 = \bmat{1 & 2 & 3 & 4 & 5 & 6}$. Consequently, a rank-1 approximation of $\mX_1$ with a total of 12 numbers is sufficient to exactly represent the data. Despite having \emph{exactly the same numbers} (albeit organized differently), the matrix $\mX_2$ is rank-4. The best rank-1 approximation 

\section{Related work}

% https://arxiv.org/pdf/1002.4545.pdf

% rephrase this... We suspect many people reading this will state: ``of course, this result is obvious.'' While we agree, the point of this paper is that the results from these two natural examples show the power of just how much a better data organization can improve the situation -- in particular on the qualitative dimension of the approximation.  The results are obviously so much better. In this section, we discuss a few related ideas that we are aware of. 

As mentioned previously, the observation that we can improve the approximation of a matrix by reorganizing the data in the SVD is not new. The point of this manuscript is that the results from these two natural examples, in images and in time-series data, show the power of just how much an improved data organization can improve the approximation -- in particular on the qualitative dimension.

We'll survey a few related ideas in no particular order and conclude with a relationship back to the Kronecker product SVD and other observations about related structure. 

\subsection{Natural organizations}
The association of an image with a matrix of data is widely understood to be \emph{natural}. However, this need not be the case. 
A related point was made in the science fiction novel \emph{Children of Time}~\cite{Tchaikovsky-2015-ChildrenOfTime} whereby a set of sentient spiders sought to communicate with a human-like intelligence. Spiders weave a two-dimensional web in a circular fashion from inwards to outwards (\url{https://www.youtube.com/watch?v=zNtSAQHNONo}) -- suggesting a completely different ``natural'' organization of pixel data. So a minor-plot point is getting the human-like intelligence to understand the radial image encoding. To the spiders, the square grid organization is not natural. 

\subsection{Better than SVD approximation}
Although not comprehensive, we note that \citet{Sui-2013-low-rank} designed a low-parameter matrix approximation that first clusters data and then computes an SVD within each cluster. This results in a higher-rank but lower-parameter approximation of a matrix than is possible with the SVD. 

\subsection{The hidden Kronecker product SVD}
In fact, what we are doing with the tiled approximation is exactly a Kronecker product SVD~\cite{VanLoan-1993-kronecker}.  A Kronecker-product SVD is a means of forming an approximation\footnote{There is ambiguity in this specific representation because we can scale each term and adjust the other. However, this has no impact on our discussion, so we try and keep things simple.}  
\[ \mA = \sum\nolimits_{j} \mB_j \kron \mC_j \] 
where the dimensions of $\mB_j$ and $\mC_j$ are the same for all $j$ and the overall dimensions of the problem match. To see that this is equivalent with our matrix of tiles approach, let's consider a specific instance that is representative of the general case. 

We need two small bits of notation. First $\tvec(\mA)$ is a vectorized representation of a matrix as a single column.\footnote{In Julia, this is simply the operation \texttt{vec} and in Matlab, it was \texttt{[:]}.} The second piece of notation is that matrices and vectors will be written as bold but elements of matrices are non-bold. 

Consider a $3 \times 4$ matrix of 12 tiles, where each tile is $k \times k$, \footnote{The tile size need not be square.} 
\[ \mA = \bmat{ \mT_1 & \mT_4 & \mT_7 & \mT_{10} \\
                \mT_2 & \mT_5 & \mT_8 & \mT_{11} \\
                \mT_3 & \mT_6 & \mT_9 & \mT_{12} } \] 
The overall matrix dimensions are $3k \times 4k$. 
In the SVD of tiles we describe above, we reshape this into a matrix $\mX$ where 
\[ \mX = \bmat{ \tvec(\mT_1) & \tvec(\mT_2) & \cdots & \tvec(\mT_{12})} . \] 
Then an SVD of $\mX$ is given by 
\[ \mX = \mU \mSigma \mV^T. \]
Let $\vu_j$ be the $j$th column of $\mU$, this gives a representation of each tile
\[ \tvec(\mT_i) = \mU \mSigma \mV^T \ve_i = \sum\nolimits_{j} \sigma_j \vu_j V_{ij}. \] 

Now consider the Kronecker product SVD 
\[ \mA = \bmat{ \mT_1 & \mT_4 & \mT_7 & \mT_{10} \\
                \mT_2 & \mT_5 & \mT_8 & \mT_{11} \\
                \mT_3 & \mT_6 & \mT_9 & \mT_{12} } 
  = \sum\nolimits_j \mB_j \kron \mC_j. \]
Here, we have $\mB_j$ is $3 \times 4$ and $\mC_j$ is $k \times k$. Index $\mB_j$ as follows 
\[ \mB_j = \bmat{ B_j^{(1)} &  B_j^{(4)}  &  B_j^{(7)}  &  B_j^{(10)}  \\[1ex]
  				  B_j^{(2)} &  B_j^{(5)}  &  B_j^{(8)}  &  B_j^{(11)}  \\[1ex]
  				  B_j^{(3)} &  B_j^{(6)}  &  B_j^{(9)}  &  B_j^{(12)}  } \] 
Then note that 
\[ \mT_i = \sum\nolimits_j B_j^{(i)} \mC_j. \]   		

The final observation that connects the two is 
\[ B_j^{(i)} = V_{ij} \text{ and } \sigma_i \vu_j = \tvec(\mC_j). \]
In both cases, we represent each tile as a specific linear combination of basis elements. In the matrix-of-vectorized tiles $\mX$, these basis elements are columns $\vu_j$; in the Kronecker product SVD, these basis elements are the matrices $\mC_j$. 

In fact, the conversion from $\mA$ to $\mX$ to compute the Kronecker product SVD is precisely what is described by~\citet{VanLoan-1993-kronecker}.  

\subsection{Kronecker product structure}

The most strongly related work is \citet{Cai2022}, which makes a similar observation about the SVD of images compared with the Kronecker product SVD approximation. More generally, Kroencker product structure is important in low-parameter / low-Kronecker rank approximations of matrices involved in neural networks~\cite{Martens2015}. More recently, \citet{Kwon2023} suggests using more complicated neural approximations instead of direct matrix products, but with similar types of structure. In a different scenario, \citet{Bamberger2022} suggests using Kronecker structure as an alternative to classic dimension reduction.

%The image as a matrix of vectorized tiles is an instance of Kronecker product approximation~\cite{VanLoan-1993-kronecker}. 

%Nevertheless, the implications seem to continually recur with comparisons of sophisticated data reshaping to the results on the straightforward SVD approximation. 
 
%Best approximation in the matrix 1-norm, and other non-orthogonally invariant norms 

%DS search candidate who went to wisconsin who did random walks/compression/etc.

%Nearest Kronecker product preconditioner 

% K-Fac Kronecekcr factorization 

\subsection{Tensor approximations}
Another place we've seen similar tiled organization is in tensor factorization studies, such as \citet{Wang2017}. In these results, an input database of images or video is reshaped into a large number of tensor modes in a fashion that is similar to the tiled approximations we describe. The result is an approximation that can use fewer parameters, but has more complex interactions. 

\section{Discussion}
This manuscript is the start of a broader discussion on this topic and is posted as a preprint to stimulate conversation. Although the case of Kronecker structure in images has been described, and the utility of Kronecker structure more generally has been noticed, the ideas of reorganization go beyond just this. For instance, the diagonal reorganization does not fit a Kronecker structure model. %Moreover, there are alternatives to the Kronecker product such as the box-product that would product alternative decompositions~\cite{XXX}. 
Finally, we note that the actual basis elements themselves produced by these approximations are useful, which is another direction we are also exploring. 

%\begin{multicol}{3}
\begin{fullwidth}
\bibliographystyle{dgleich-bib}
%\bibliography{../../../bibliography/all-bibliography,
\bibliography{refs}
\end{fullwidth}
%\end{multicol}{3}

%\subsection{COVID positivity details}
%\label{sec:details-covidpos}
%We accessed the data on COVID-19 positivity rates from \url{https://api.covidtracking.com/v1/states/daily.csv}. Then we divided the total number of positive results by the total number of tests to get a positivity rate. We then formed a 7 day moving average of measurements starting from May 17, 2020 until October 13, 2022 (150 days) for each of the 50 US states for 150 days. (The 7-day moving averages on May 17 reflects May 11-May 17.) 
%
%The plot is normalized to show the maximum dynamic range for each state.  T
%
%\subsubsection{Ecotyper data}
%We downloaded the bulk lung data from \url{https://raw.githubusercontent.com/digitalcytometry/ecotyper/master/example_data/bulk_lung_data.txt}
%

\end{document}